\documentclass[12pt, reqno]{amsart}

\usepackage{amscd,amsthm,amsfonts,amssymb,amsmath,esint}
\usepackage{amsmath,tikz-cd}
\usepackage[colorlinks=true]{hyperref}
\usepackage{fullpage}
\usepackage[all]{xy}
\usepackage{mathtools}
\usepackage{mathrsfs}
\usepackage[pagewise]{lineno}
\usepackage{cite}
    \newcommand{\BC}{{\mathbb {C}}}

    \newcommand{\BQ}{{\mathbb {Q}}} 
    \newcommand{\BS}{{\mathbb {S}}}

     \newcommand{\BZ}{{\mathbb {Z}}}

    \newcommand{\RO}{{\mathrm {O}}}

    \newcommand{\fa}{{\mathfrak{a}}}

     \newcommand{\fh}{{\mathfrak{h}}}

     \newcommand{\fn}{{\mathfrak{n}}}

    \newcommand{\fS}{{\mathfrak{S}}}

    \newcommand{\Frob}{{\mathrm{Frob}}}

     \newcommand{\GL}{{\mathrm{GL}}}

    \newcommand{\Ind}{{\mathrm{Ind}}}

    \newcommand{\ord}{{\mathrm{ord}}}

    \renewcommand{\Re}{{\mathrm{Re}}}

    \newcommand{\Res}{{\mathrm{Res}}}

    \newcommand{\SL}{{\mathrm{SL}}}

    \newcommand{\Sym}{{\mathrm{Sym}}}

    \theoremstyle{plain}
    \newtheorem{thm}{Theorem}[section] \newtheorem{cor}[thm]{Corollary}
    \newtheorem{lem}[thm]{Lemma}  \newtheorem{prop}[thm]{Proposition}

 \newtheorem{def-prop}[thm]{Definition-Proposition}

\theoremstyle{remark} \newtheorem{remark}[thm]{Remark}
\theoremstyle{remark} 
\theoremstyle{remark} \newtheorem*{example}{Example}

    \numberwithin{equation}{section}

\begin{document}

\title{Higher moments for symmetric powers of modular forms}

\begin{abstract}
    Let $f$ be a cuspidal eigenform of weight $k$ on $\SL_2(\BZ)$ and let $\lambda_{\Sym^d f}(n)$ be the normalized Fourier coefficients of its $d$-th symmetric power lift. This paper establishes asymptotic formulas for the moments $\sum_{n\leq x}\lambda^l_{\Sym^d f}(n)$ for all positive integers $d$ and $l$. We also prove an asymptotic formula for the corresponding sum over the values of any positive definite binary quadratic form $Q$. Our results generalize and improve upon previous work, which was limited to small values of $d$ or  $l$. The proofs rely on the decomposition of $\ell$-adic Galois representations and the analytic properties of the associated $L$-functions. 
    \end{abstract}

\author{Jiong Yang and Zhishan Yang}
\thanks{The first author  is supported by the Natural Science Foundation of Shandong Province (NO.ZR2022QA097).}
\address{School of Mathematics and Statistics, Qingdao University, Qingdao, Shandong, 266000, People's Republic of China; }
\email{yangjiong@qdu.edu.cn}
\address{School of Mathematics and Statistics, Qingdao University, Qingdao, Shandong, 266000, People's Republic of China }
\email{zsyang@qdu.edu.cn}

\maketitle

\tableofcontents

\section{Introduction}
%In this paper, we estimate the average behavior of various of higher moments associated to  Fourier coefficients of modular forms.

Let $f$ be a cuspidal  eigenform of weight $k$ on $\SL_2(\BZ)$,  and let
\[f(z)=\sum_{n=1}^\infty \lambda_f(n)n^{\frac{k-1}{2}}q^n\]
be its $q$-expansion. Estimates of $\lambda_f(n)$ and their average behavior are fundamental problems in number theory. 
%By Deligne's work on Ramanujan's conjecture (\cite{deligne1974conjecture}), one has
%\[\lambda_f(n)\leq d(n)\]
%for all $n\geq 1$ with $d(n)$ the divisor function.
 Estimating the sum
$\sum\limits_{n\leq x}\lambda_f(n),$ has a long history, with contributions from
 \cite{kloosterman1927asymptotische}, \cite{rankin1983sums}, \cite{weil1948some}, \cite{wu2009power}, among others. The sum of higher power moments 
$\sum\limits_{n\leq x}\lambda_f^l(n)$
has also attracted considerable attention. If $l=2$, Rankin\cite{rankin1939contributions} and Selberg\cite{selberg1940bemerkungen} already studied this problem and recently Huang's work \cite{huang2021rankin} provides improved estimates with
\[\sum_{n\leq x}\lambda_f^2(n)=Cx+O\left(x^{\frac{3}{5}-\frac{1}{560}+\epsilon}\right).\]

For $l\geq3$,  symmetric power $L$-functions of modular forms are required.  Kim and  Kim \& Shahidi   \cite{kim2003functoriality}, \cite{kim2002cuspidality},\cite{kim2002functorial} showed that  $\Sym^l f$ is automorphic for $l\leq 4$. Using these results,  the estimates for $l\leq 8$  are derived in  \cite{lau2011integral}, \cite{lu2009average} and \cite{lu2011higher}. Later Zhai\cite{zhai2013average} considered the sum over sums of two squares
\[\sum_{n_1^2+n_2^2\leq x}\lambda_f^l(n_1^2+n_2^2),\]
and provided estimates for $l\leq 8$.

Recently, Newton \& Thorne \cite{newton2021symmetric1},\cite{newton2021symmetric2} proved the automorphy and cuspidality of $\Sym^l f$ for all $l\geq1$. Using this celebrated work, Xu\cite{xu2022general} derived a general formula for $\sum\limits_{n\leq x}\lambda_f^l(n)$ and $\sum\limits_{n_1^2+n_2^2\leq x}\lambda_f^l(n_1^2+n_2^2)$ for every $l\geq1$. This result was later refined by Liu\cite{liu2023asymptotic} as
\begin{equation}\label{equation:introduction1}
  \sum_{n\leq x}  \lambda^l_f(n)=xQ_l(\log x)+O\left(x^{\theta_l+\epsilon}\right),
\end{equation}
\begin{equation}\label{equation:introduction2}
 \sum_{n_1^2+n_2^2\leq x}\lambda_f^l(n_1^2+n_2^2)=xQ'_l(\log x)+O\left(x^{{\theta}'_l+\epsilon}\right),
\end{equation}
where $Q_l$ and $Q'_l$ are polynomials of degree $\binom{l}{m}-\binom{l}{m-1}-1$  if $l=2m$ and equals 0 if $l$ is odd. Here $\theta_l, {\theta}'_l$ are explicit constants less than 1.

Another direction of research concerns the average behavior of  Fourier coefficients of symmetric powers of $f$. Several results have been obtained in this direction, but  they only focus on some special cases. For example,  in \cite{fomenko2006identities} and \cite{fomenko2008mean},  Fomenko proved asymptotic formulas for 
\[\sum_{n\leq x}\lambda_{\Sym^2 f}(n)\ \text{and}\ \sum_{n\leq x}\lambda^2_{\Sym^2 f}(n).\]
In \cite{he2019integral} and \cite{luo2021asymptotics},  the summations \[\sum_{n\leq x}\lambda^l_{\Sym^2 f}(n)\ \text{and}\ \sum_{n\leq x}\lambda^2_{\Sym^d f}(n)\]
with $d\geq 2$ and $l\leq 8$ were studied. Later  in \cite{liu2023average}, Liu obtained the following results with better error terms:
\begin{equation} \label{equation:introduction3}
    \sum_{n\leq x}\lambda^l_{\Sym^2 f}(n)=xP_l(\log x)+O\left(x^{\theta_l+\epsilon}\right)
\end{equation}
for $2\leq l\leq 8$, and
\begin{equation}\label{equation:introduction4}
    \sum_{n\leq x}\lambda^2_{\Sym^d f}(n)=c_dx+O\left(x^{\tilde{\theta}_d+\epsilon}\right),
\end{equation}
for any $d\geq 2$.
Here $P_l$ are polynomials with  explicitly computed degrees, $c_d$ is a constant and $\theta_l, \tilde{\theta}_d$ are positive constants less than 1.

Extending these results to arbitrary integers $d\geq1$ and $l\geq1$ presents significant challenges, as the associated Dirichlet series become increasingly complex. In this paper, we address this problem by studying the underlying Galois representations and applying combinatorial results concerning their decompositions. 
Our work generalizes the average formulas (\ref{equation:introduction1}),(\ref{equation:introduction2}),(\ref{equation:introduction3}),(\ref{equation:introduction4}) in the following aspects.
\begin{itemize}
    \item We obtain asymptotic formulas with improved error terms.
    \item We establish a universal asymptotic formula for $\sum\limits_{n\leq x}\lambda^l_{\Sym^d f}(n)$ valid for any $d\geq 1$  and $l\geq 1$,which generalizes (\ref{equation:introduction3}) and (\ref{equation:introduction4}).
        \item We prove an average formula for sums over values of a binary quadratic form, namely $ \sum\limits_{Q(n_1,n_2)\leq x}\lambda^l_{\Sym^d f}(Q(n_1,n_2))$ for any definite binary quadratic form $Q$ and any $d,l\geq1$.
\end{itemize}
\subsection{Main theorems}
We now state our main results for general $d$ and $l$.
\begin{thm}\label{theorem:main}
    Let $f$ be a cuspidal modular form of weight $k$ on $\SL_2(\BZ)$. For any integers $l\geq2$ and $d\geq1$ and $dl>4$, one has
    \[\sum_{n\leq x}  \lambda^l_{\Sym^d f}(n)=xP_{d,l}(\log x)+O\left(x^{\theta_{d,l}+\epsilon}\right).\]
    Here $P_{d,l}$ is a polynomial of degree $K_{0,d,l}-1$, where $K_{i,d,l}$ denote the Kostka number defined in sections \ref{section:Weylmod} and \ref{section:Kostka number}. The exponent $\theta_{d, l}$ is given by
    $$\theta_{d,l}=1-1/\left((d+1)^l/2-4K_{0,d,l}/21-K_{1,d,l}/3-5K_{2,d,l}/14\right).$$
\end{thm}
In Corollary \ref{corollary:Kostka},  we  provide an explicit formula for the Kostka number $K_{i,d,l}$, which determines the degrees of the main terms and the exponents in the error terms. The formula is 
\begin{align*}        K_{i,d,l}=\begin{cases}
            0\ \ \ \ \ \ \ \ \ \ \ \ \ \ \ \ \ \ \ \ \ \ \ \ \ \ \  \ \ \ \ \ \ \   \ \ \ \ \ \ \ \ \ \ \ \ \text{if $2\nmid dl-i$}\\
            \sum\limits_{j=0}^{\lfloor(dl-i)/(2d+2)\rfloor}(-1)^j\binom{l}{j}\binom
        {(dl-i)/2-j(d+1)+l-2}{l-2}
    \ \text{otherwise}
        \end{cases}.
    \end{align*}
Note that  if $dl$ is odd, then $K_{0,d,l}=0$ and $P_{d,l}$ is 0.

The estimation in the general case is new. And our estimates in the error terms are better than those in the work of \cite{liu2023average} for small $d$ and $l$. We list the cases with $d=2,3\leq l\leq8$ and $l=2,3\leq d\leq8$ in the following tables.

\begin{center}
%\caption{Previous exponents, $\theta_{l, j}$ from Theorem \ref{theorem:main} }
\begin{tabular}{|c|c|c|}
\hline
$d=2, l=$ &  Exponents in \cite{liu2023average} & $\theta_{d, l}$ \\
\hline
3 & 0.9193... & 0.918287938
 \\
\hline
4 & 0.9737... & 0.973534972
  \\
\hline
5 & 0.99136... & 0.991304348
\\
\hline
6 & 0.99714... & 0.99713291
 \\
\hline
7 & 0.9990558... & 0.999051362
 \\
\hline
8 & 0.9996868... & 0.999685565
  \\
\hline
\end{tabular}
\end{center}

\begin{center}
\begin{tabular}{|c|c|c|c|}
\hline
$l=2, d=$ & Exponents in \cite{liu2023average} & $\theta_{d, l}$ \\
\hline
3 & 0.866... & 0.865814696
 \\
\hline
4 & 0.9166... & 0.916334661
  \\
\hline
5 & 0.9428... & 0.942701228
 \\
\hline
6 & 0.9583... & 0.958250497
 \\
\hline
7 & $0.96824\ldots$ &0.968205905
  \\
\hline
8 & 0.97499... & 0.974970203
 \\
\hline
\end{tabular}
\end{center}

We also have the following theorem for summation over binary quadratic forms.
\begin{thm}\label{thm:binary}
      Let $f$ be a cuspidal modular form of weight $k$ on $\SL_2(\BZ)$. For any integers $l\geq2,d\geq1$ with $dl>4$ and any binary quadratic form $Q(n_1,n_2)$, one has
       \[\sum_{Q(n_1,n_2)\leq x}  \lambda^l_{\Sym^d f}(Q(n_1,n_2))=xP_{d,l,Q}(\log x)+O\left(x^{\theta_{d,l,Q}+\epsilon}\right).\]
        Here $P_{d,l,Q}$ is a polynomial of degree $K_{0,d,l}-1$.  If the class number of $Q$ is not 1,  then 
        $$\theta_{d,l,Q}=1-\frac{3}{3(d+1)^l-K_{0,d,l}};$$ if the class number is 1, then 
        $$\theta_{d,l,Q}=1-1/\left((d+1)^l-8K_{0,d,l}/21-2K_{1,d,l}/3-5K_{2,d,l}/7\right).$$
\end{thm}

%\begin{remark}
 %   Although the main theorems of our work are stated for cudpidal modular forms on $\SL_2(\BZ)$ only, all the results and proofs work exactly the same for cuspidal modular forms   of congruence groups with real coefficients  without CM. We choose to work on $\SL_2(\BZ)$ to simplify the notations.
%\end{remark}
\subsection{Proof Sketch }
The proofs of our theorems follow a framework that can be viewed as an $\ell$-adic analogue of the methods for Artin representations developed in \cite{Yang2024second} , \cite{Yang2024first} and \cite{Yang2026higher}.

The starting point is Perron's formula, which translates the estimation of the sum $\sum\limits_{n\leq x}\lambda^l_{\Sym^d f}(n)$ into the problem of evaluating a contour integral of the associated Dirichlet series $\sum\limits_{n=1}^\infty\frac{\lambda^l_{\Sym^d f}(n)}{n^s}$. This Dirichlet series is intimately related via its Euler product to the $L$-function $L\left(\left(\Sym^d f\right)^{\otimes l},s\right)$, which corresponds to  the $\ell$-adic representation $\Sym^d(\rho_f)^{\otimes l}$ (see Proposition \ref{proposition:dirichletseriescomp}).

A key insight is the decomposition of this high-dimensional tensor representation into a direct sum of irreducible constituents, which are shown to be of the form $\Sym^i \rho_f$. Consequently, the $L$-function factors into a product of symmetric power $L$-functions:
$$L\left(\left(\Sym^d f\right)^{\otimes l},s\right)=\prod_{i=0}^{ld}L\left(\Sym^i f,s\right)^{K_{i,d,l}}$$
The automorphy and entirety of these symmetric power $L$-functions, established by Newton and Thorne \cite{newton2021symmetric1},\cite{newton2021symmetric2}, are crucial here. The final asymptotic formulas are then derived by combining this factorization with established subconvexity bounds for the involved $L$-functions.

For sums over values of a binary quadratic form $Q$, the approach is similar but involves additional ingredients. The generating Dirichlet series becomes $\sum\limits_{n=1}^\infty\frac{\lambda^l_{\Sym^d f}(n)r(n,Q)}{n^s}$, where $r(n,Q)$ counts representations of $n$ by $Q$. Let $F=\BQ(\sqrt{D})$ be the corresponding imaginary quadratic field. When the class number $h(D)=1$, the series $\sum\limits_{n=1}^\infty \frac{r(n,Q)}{n^s}$ is essentially the Dedekind zeta function $\zeta_F(s)$.  In general, it can be expressed as a finite sum of Hecke $L$-functions associated to characters of the class group of $F$. Each such $L$-function corresponds to a modular form $f_\chi$
 of weight $1$. Thus, the analysis reduces to studying Rankin-Selberg type $L$-functions of the form $L\left(\left(\Sym^d f\right)^{\otimes l}\otimes\rho_{f_\chi},s\right)$ which again admit a decomposition into products of simpler $L$-functions. The theorem is proved by applying subconvexity estimates to these components.

\begin{remark}
    While this paper has been submitted and review, we noticed that recently Venkatasubbareddy submitted a paper\cite{arXiv:2601.17079} working on a similar problem and obtained similar results as our Theorem \ref{theorem:main}. But our work is different from \cite{arXiv:2601.17079} in the following aspects. First we use a representation  point of view to deal with the Dirichlet series, this allows us to connect our work with combination. Such ideas are similar with our work \cite{Yang2024first} and \cite{Yang2026higher}. They can be used to deal with more complicated cases and play a key role in the following work \cite{Yang2026second}. Next we also consider average results over binary quadratic forms, which are not considered in \cite{arXiv:2601.17079}. Finally, the estimation for the error terms in the two works are different.

\end{remark}

\section{$L$-functions associated to Weyl modules of modular forms}
\subsection{$L$-functions of compatible systems of $\ell$-adic representations}
Fix an integer $n\geq1$. For each prime $\ell$, let $\rho_\ell:G_\BQ\to\GL(V_\ell)\simeq\GL_n(\bar{\BQ}_\ell)$ be an $\ell$-adic Galois representation. 
%A system of $\ell$-adic representations $\{\rho_\ell\}_\ell$ is called compatible if
%\begin{itemize}
%    \item[(1)] There is a finite set $S$ of primes so that $\rho_\ell$ is unramified at all $p\notin S\cup\{\ell\}$
%    \item[(2)] For each prime $p$, the polynomial $\det(1-T\rho_\ell(\Frob_p)|V_\ell^{I_\ell})\in\BQ_\ell[T]$ lies in $\BQ[T]$ and does not depends on $\ell\neq p$, where $\Frob_p$ is a Frobenius element at $p$ and $I_\ell$ is the inertial group at $\ell$.
%\end{itemize}
If $\rho=\{\rho_\ell\}$ is a compatible system of $\ell-$adic representations, then we can associate an $L$-function $L(\rho,s)$ to it.
%then define the local $L$-factor at $p$ as
%\[L_p(\rho,T)=\frac{1}{\det(1-T\rho_\ell(\Frob_p)|V_\ell^{I_\ell})}\]
%for any $\ell\neq p$  and define the $L$-function as
%\[L(\rho,s)=\prod_pL_p(\rho,p^{-s}).\]

Let $\rho_1,\rho_2$ be two compatible systems. Then 
\begin{equation}
    L(\rho_1\oplus \rho_2,s)=L(\rho_1,s)L(\rho_2,s).
\end{equation}

One of the most important examples of $\ell$-adic Galois representations arises from modular forms.
Let $f$ be a cuspidal modular form of weight $k$ for $\SL_2(\BZ)$. By the work of Deligne, we can associate  a two-dimensional compatible system of $\ell-$adic representations $\rho_f=\{\rho_{f,\ell}\}$ to $f$. 
Let $f(z)=\sum\limits_{n=1}^\infty \lambda(n)n^{\frac{k-1}{2}}q^n$ be the normalized $q$-expansion of $f$. For any prime number $p$, let $\alpha_f(p),\beta_f(p)$ denote the normalized Hecke eigenvalues of $f$ at $p$, satisfying
\[\alpha_f(p)+\beta_f(p)=\lambda_f(p),\ \alpha_f(p)\beta_f(p)=1.\]
Define the automorphic $L$-function of $f$  as
\[L(f,s)=\sum_{n=1}^\infty \frac{\lambda_f(n)}{n^s}=\prod_p (1-\lambda_f(p)p^{-s}+p^{-2s})^{-1}.\]
Deligne's work shows that  $\alpha_f(p)$ and $\beta_f(p)$ are also the eigenvalues of $\rho_\ell(\Frob_p)$ on $V_\ell^{I_\ell}$.
Thus the two types of $L$-functions coincide:
\[L(f,s)=L(\rho_f,s).\]

We can also consider the symmetric power of a modular form. For any $d\geq 1$,  let $\Sym^{d}\rho_f$ denote the symmetric power representation of $\rho_f$. The modularity of $\Sym^{d}\rho_f$ is obtained in \cite{newton2021symmetric1} and \cite{newton2021symmetric2}, and we have
\[L\left(\Sym^d f,s\right)=L\left(\Sym^d\rho_f,s\right),\]
where  $\Sym^d f$ is the symmetric power lifting of $f$ to $\GL_{d+1}$. 
Thus the $L$-function $L(\Sym^d\rho_f,s)$ is entire on the whole complex plane. 
%There are several common ways to obtain new compatible systems of $\ell$-adic representations from  known ones. In this paper, we are interested in the direct sum, tensor product and symmetric power representations. The formula for direct sum representation is simple. 
%As for other constructions, we will consider a simple case arising from modular forms.

For $i=1,2$, let $\rho_i=\{\rho_{i,\ell}\}$ be compatible systems of $\ell-$adic representations. Consider the Dirichlet series 
\[D\left(\rho_1\otimes\rho_2,s\right):=\sum_{n=1}^\infty \frac{\lambda_{\rho_1}(n)\lambda_{\rho_2}(n)}{n^s},\ D\left(\Sym^d\rho_i,s\right):=\sum_{n=1}^{\infty}\frac{\lambda_{\rho_i}(n^d)}{n^s}.\]
\begin{thm}\label{theorem:l-functions}
    We have the following relations 
    \begin{equation}
        L\left(\rho_1\otimes\rho_2,s\right)=D\left(\rho_1\otimes\rho_2,s\right)U_1(s),
    \end{equation}
    \begin{equation}
        L\left(\Sym^d\rho_i,s\right)=D\left(\Sym^d\rho_i,s\right)U_2(s),
    \end{equation}
    where $U_1$ and $U_2$ are Euler products absolutely convergent for $\Re(s)>1/2$ and uniformly convergent in the region $\Re(s)>1/2+\epsilon$ for any $\epsilon>0$.
\end{thm}
\begin{proof}
    The proof follows the same lines as the case  of Artin representations given in \cite{Yang2024second}  and  \cite{Yang2024first}.
\end{proof}
\subsection{The Schur functor and Weyl modules}
In this section, we introduce the properties of Weyl modules used in this paper. For further details, we refer to \cite{fulton2013representation}.
Let $d$ be a positive integer, and let  $\fS_d$ be the symmetric group.  
Let $d=\lambda_1+\lambda_2+\cdots+\lambda_t$ with $\lambda_1\geq\lambda_2\geq\cdots\geq\lambda_t\geq 0$ be a partition of $d$.  Such a partition is denoted by $\lambda=(\lambda_1,\lambda_2,\cdots,\lambda_t)$ as usual. We identify two partitions if they differ only by trailing zeros. 
%Let $\lambda,\mu$ be two partitions, then we say $\lamba>\mu$ if the first nonvanishing $\lambda_i-\mu_i$ is positive.

 Let $\lambda=(\lambda_1,\cdots,\lambda_t)$ be a partition   of $d$. We associate a Young symmetrizer $c_\lambda\in\BC\fS_d$  to $\lambda$ . Let $V$ be an $\ell$-adic representation of dimension $m$. This $c_\lambda$ acts on $V^{\otimes d}$ and its image is also an $\ell$-adic representation,  denoted by $\BS_\lambda V$ . These  representations are called the Weyl modules of $V$.
For example, if $\lambda=(d)$, then $\BS_\lambda V=\Sym^d V$ is the $d$-th  symmetric power representation, and if $\lambda=(1,1,\cdots,1)$, then $\BS_\lambda V=\Lambda^d V$ is the wedge product of $V$. 

We have the following results.
\begin{prop}\label{prop:dimension}
Let $V$ be an $\ell$-adic representation of dimension $m$.
    \begin{itemize}
        \item[(1)] If $\lambda_i\neq 0$ for some $i>m$, then $\BS_\lambda V=0$. 
        \item[(2)] $\dim \BS_\lambda V=\prod_{1\leq i<j\leq m}\frac{\lambda_i-\lambda_j+j-i}{j-i}$. In particular, $\dim \Sym^d V=\binom{m+d-1}{d}$.
        %\item[(3)] The  traces of an element $g$ in $\GL(V)$ on $\BS_\lambda V$ can be computed as
    \end{itemize}
\end{prop}

A fundamental problem is  describing the decomposition of tensor products
 of Weyl modules.  Let $\lambda$ be a partition of $d_1$, and $\mu$ be a partition of $d_2$, then
\begin{equation}
    \BS_\lambda V\otimes \BS_\mu V\cong\bigoplus_{\nu}N_{\lambda\mu\nu}\BS_\nu V.
\end{equation}
Here $\nu$ runs over all partitions of $d_1+d_2$, and $N_{\lambda\mu\nu}$ are non-negative integers determined by the Littlewood-Richardson rule. This  rule states that  $N_{\lambda\mu\nu}$ is the number of  ways to extend the Young diagram of $\lambda $ to that of $\nu$ via a strict $\mu$-extension.

In particular,   the tensor power $V^{\otimes d}$ decomposes as
\begin{equation}\label{equation:tensordec}
    V^{\otimes d}\cong\bigoplus_\lambda \BS_\lambda V^{\oplus m_\lambda}.
\end{equation}
Here $\lambda$ runs over all partitions of $d$ and $m_\lambda$ is the dimension of the irreducible representation of $\fS_d$ obtained by right multiplication of $c_\lambda$ on $\BC\fS_d$.
This $m_\lambda$ can be computed via
\begin{equation}
   m_\lambda=\frac{d!}{l_1!\cdots l_k!}\prod_{i<j}(l_i-l_j), 
\end{equation}
where $k$ is the number of rows of the Young diagram of $\lambda$ and $l_i=\lambda_i+k-i$.

Another special case is the tensor product of symmetric power representations. Assume  $i\geq j\geq 0$, then 
\begin{equation}\label{equation:symtensor}
    \Sym^i V\otimes \Sym^j V\cong\bigoplus_{t=0}^j\BS_{(i+t,j-t)} V.
\end{equation}
In general, we have the isomorphism
\begin{equation}\label{equation:symtendec}
  \Sym^{\lambda_1} V\otimes\Sym^{\lambda_2}V\otimes\cdots\otimes\Sym^{\lambda_k}V\cong\bigoplus_\mu (\BS_\mu V)^{\oplus K_{\mu,\lambda}}. 
\end{equation}
Here $\mu$ runs over all  partitions of $d=\lambda_1+\lambda_2+\cdots+\lambda_k$, and $K_{\mu,\lambda}$ is the Kostka number. There are several ways to define the Kostka number. For instance, it is the number of semistandard tableaux of shape $\mu$ and type $\lambda$. This means the number of  ways to fill the boxes of the Young diagram of $\mu$ with $\lambda_1$ copies of $1$, $\lambda_2$ copies of $2$, ..., $\lambda_k$ copies of $k$, such that the entries in  each row are weakly increasing and entries in each column are strictly increasing. In general, there is no explicit formula for $K_{\mu,\lambda}$ for arbitrary $\mu$ and $\lambda$.
%We also have a formula for the Weyl module of direct sums of representations. Here we only recall the formula for symmetric power representations.
%\[\Sym^d(V\oplus W)=\oplus_{a+b=d}(\Sym^a V\otimes\Sym^b W).\]

\subsection{Weyl modules associated to modular forms}\label{section:Weylmod}
Now we specialize to the case  of modular forms.
Let $f$ be a cuspidal eigenform of weight $k$ on $\SL_2(\BZ)$ and let $\rho_f$ be its associated $\ell$-adic Galois representation. 
\begin{prop}\label{prop:decomposition}
\

    \begin{itemize}
        \item[(1)] If the  Young diagram of $\lambda$ has more than two rows, then $\BS_\lambda \rho_f=0$. Moreover $\dim \Sym^d \rho_f=d+1$.
        \item[(2)] If $\lambda=(\lambda_1,\lambda_2)$ is a partition of $d$, then $\BS_{\lambda} \rho_f\cong \Sym^{\lambda_1-\lambda_2}\rho_f$.
        \item[(3)] $\rho_f^{\otimes l}\cong\bigoplus\limits_{i=0}^{\lfloor l/2\rfloor} \Sym^{l-2i} \rho_f^{\oplus c_{l,i}}$. Here $c_{l,i}=\binom{l}{i}-\binom{l}{i-1}$.
        \item[(4)] $\Sym^i\rho_f\otimes \Sym^j \rho_f\cong\Sym^{|i-j|} \rho_f\oplus \Sym^{|i-j|+2}\rho_f\oplus\cdots\oplus\Sym^{i+j}\rho_f$.
           
        \item[(5)] $(\Sym^d \rho_f)^{\otimes l}\cong\mathop{\oplus}\limits_{i=0}^{dl} K_{i,d,l}\Sym^i \rho_f$, where  $K_{i,d,l}=0$ if $i$ and $dl$ have different parity, and  $K_{i,d,l}=K_{\mu,\lambda}$ (the Kostka number)  otherwise. Here $\mu=(\frac{dl+i}{2},\frac{dl-i}{2})$ and $\lambda=(d,d,\cdots,d)$ are partitions of $dl$.
        \end{itemize} 
\end{prop}
\begin{proof}
    The statements in (1) follow directly from Proposition \ref{prop:dimension}, since  $\rho_f$ has dimension $2$.
    
    Let $\alpha_f,\beta_f$ be the eigenvalues of $\Frob_p$ acting on $\rho_f$, with eigenvector $v_1,v_2$. 
    For any $j\geq 1$, $\Sym^j\rho_f$ has a basis $v_{i_1}\otimes v_{i_2}\otimes\cdots\otimes v_{i_j}$ with $1\leq i_1\leq i_2\leq\cdots\leq i_j\leq2$. The eigenvalues of $\Frob_p$ acting on $\Sym^j\rho_f$ are  $\alpha_f^{t_1}\beta_f^{t_2}$ for $t_1+t_2=j$.
   For $\BS_\lambda\rho_f$, a basis is given by  the image of   $w_1\otimes w_2\otimes\cdots\otimes w_d\in \rho_f^{\otimes d}$, where
  $w_i=v_1$ or $v_2$ and $w_i\neq w_{d-i}$ for any $1\leq i\leq \lambda_2$.
 The eigenvalue of $\Frob_p$  acting on these elements are  $(\alpha_f\beta_f)^{\lambda_2}\alpha_f^{t_1}\beta_f^{t_2}=\alpha_f^{t_1}\beta_f^{t_2}$, where $t_1+t_2=\lambda_1-\lambda_2$. The statement in (2) follows by comparing  Frobenius eigenvalues.

 The satements in (3), (4) and (5) follow from the decomposition formulas (\ref{equation:tensordec}) , (\ref{equation:symtensor}) and (\ref{equation:symtendec})  respectively, together with (2).
\end{proof}

From the decomposition of the tensor product of  representations, we derive the factorization of $L$-functions.
\begin{prop}
           \[L\left(f^{\otimes l},s\right)=\prod_{i=0}^{\lfloor l/2\rfloor}L\left(\Sym^{l-2i}f,s\right)^{c_{l,i}},\ \ L\left(\left(\Sym^d f\right)^{\otimes l},s\right)=\prod_{i=0}^{ld}L\left(\Sym^i f,s\right)^{K_{i,d,l}}.\]
\end{prop}
The first identity was also derived in \cite{xu2022general} by comparing  Hecke eigenvalues on these representations.

\subsection{The calculation of Kostka numbers}\label{section:Kostka number}
In this section, we provide an explicit formula for the Kostka number $K_{i,d,l}$ appearing in the decomposition of $\left(\Sym^d V\right)^{\otimes l}$ in Proposition \ref{prop:decomposition}.

For any $0\leq i\leq dl$ such that $i$ and $dl$ have the same parity,
let $\lambda=(d,d,\cdots,d),\mu=(\frac{dl+i}{2},\frac{dl-i}{2})$ be two partitions of $dl$. The Kostka number $K_{\mu,\lambda}$ is the number of ways to fill the Young tableaux of shape $\mu$ with the numbers $1,2,\cdots,l$, where each number appears exactly $d$ times, such that entries in each row are weakly increasing and entries in each column are strictly increasing.

We now provide an alternative explicit description of these Kostka numbers. Let $K_{i,d,l}=K_{\mu,\lambda}$ (as defined above).  By Proposition \ref{prop:decomposition}, $K_{i,d,l}$ is the multiplicity of $\Sym^i V$ in $\left(\Sym^d V\right)^{\otimes l}$ for a two dimensional representation $V$. We extend the definition of $K_{i,d,l}$  to all $i\in\BZ$ by setting $K_{i,d,l}=0$ for $i\notin [0,dl]$, or when $i$ and $dl$ have different parity.  From Proposition \ref{prop:decomposition} (4), we  derive a recursive formula for $K_{i,d,l}$:
\[K_{i,d,l}=K_{i-d,d,l-1}+K_{i-d+2,d,l-1}+\cdots+K_{i+d,d,l-1}.\]
On the other hand, consider the polynomial:
\[(1+x+x^2+\cdots+x^d)^{l}=\sum_{j=-\infty}^{+\infty}C_{j,d,l}x^j.\]
Define $A_{i,d,l}=C_{\lfloor\frac{dl-i}{2}\rfloor,d,l}-C_{\lfloor\frac{dl-i-1}{2}\rfloor,d,l}.$
\begin{prop} We have 
   $ K_{i,d,l}=A_{i,d,l}.$
 \end{prop}
\begin{proof}
We only need to show that $A_{i,d,l}$ and $K_{i,d,l}$ satisfy the same recursive formula and  boundary conditions.
   First, the coefficients $C_{i,d,l}$ satisfy the recursion:
    \[C_{i,d,l}=C_{i-d,d,l-1}+C_{i-d+1,d,l-1}+\cdots+C_{i,d,l-1}.\]
    Thus
    \begin{align*}
         A_{i,d,l}=&C_{\lfloor\frac{dl-i}{2}\rfloor,d,l}-C_{\lfloor\frac{dl-i-1}{2}\rfloor,d,l}\\
         =&\sum_{t=j-d}^jC_{t,d,l-1}-\sum_{t=j'-d}^{j'}C_{t,d,l-1}.
    \end{align*}
   % \begin{align*}
   %     A_{i,d,l}=&C_{\lfloor\frac{dl-i}{2}\rfloor,d,l}-C_{\lfloor\frac{dl-i-1}{2}\rfloor,d,l}=C_{\lfloor\frac{ld-i}{2}\rfloor-d,d,l-1}+C_{\lfloor\frac{dl-i}{2}\rfloor-d+1,d,l-1}+\cdots C_{\lfloor\frac{dl-i}{2}],d,l-1}\\
   %     &-C_{\lfloor\frac{dl-i-1}{2}\rfloor-d,d,l-1}-C_{\lfloor\frac{dl-i-1}{2}\rfloor-d+1,d,l-1}-\cdots C_{\lfloor\frac{dl-i-1}{2}\rfloor,d,l-1}.
   % \end{align*}
    where $j=\lfloor\frac{dl-i}{2}\rfloor, j'=\lfloor\frac{dl-i-1}{2}\rfloor.$
    On the other hand,
    \begin{align*}
        \sum_{t=0}^dA_{i-d+2t,d,l-1}&=\sum_{t=0}^d\left(C_{\lfloor\frac{dl-d-(i-d+2t)}{2}\rfloor,d,l-1}-C_{\lfloor\frac{dl-d-(i-d+2t)-1}{2}\rfloor,d,l-1}\right)\\
        &=\sum_{t=0}^d\left(C_{\lfloor\frac{dl-i-2t)}{2}\rfloor,d,l-1}-C_{\lfloor\frac{dl-i-2t-1}{2}\rfloor,d,l-1}\right)
    \end{align*}
    %\begin{align*}
    %    &A_{i-d,d,l-1}+A_{i-d+2,d,l-1}+\cdots+A_{i+d,d,l-1}= C_{[\frac{dl-d-i+d}{2}],d,l-1}-C_{[\frac{dl-d-i+d-1}{2}],d,l-1} \\
    %    &+C_{[\frac{dl-d-i+d-4}{2}],d,l-1}-C_{[\frac{dl-d-i+d-5}{2}],d,l-1}+\cdots+C_{[\frac{dl-d-i-d}{2}],d,l-1}-C_{[\frac{dl-d-i-d-1}{2}],d,l-1}.
    %\end{align*}
    By telescoping, this sum equals $A_{i,d,l}$.
    Thus
   \[A_{i,d,l}=A_{i-d,d,l-1}+A_{i-d+2,d,l-1}+\cdots+A_{i+d,d,l-1}.\]
   For the boundary condition $l=1$:
   \begin{itemize}
       \item $K_{d,d,1}=1$ and $K_{i,d,1}=0$ for $i\neq d$
       \item $A_{d,d,1}=C_{\lfloor \frac{d-d}{2}\rfloor,d,1}-C_{\lfloor \frac{d-d-1}{2}\rfloor,d,1}=C_{0,d,1}-C_{-1,d,1}=1-0=1$ and $A_{i,d,1}=0$ for $i\neq d$.
   \end{itemize}
     Thus $A_{i,d,d}$ and $K_{i,d,d}$ satisfy the same recursion  and boundary conditions, so $A_{i,d,l}=K_{i,d,l}$.
\end{proof}
From this proposition, the Kostka number $K_{i,d,l}$ can be expressed in terms of binomial coefficients. This gives the following formula.
\begin{cor}\label{corollary:Kostka}
    \begin{equation}
        K_{i,d,l}=\begin{cases}
            0\ \ \ \ \ \ \ \ \ \ \ \ \ \ \ \ \ \ \ \ \ \ \ \ \ \ \  \ \ \ \ \ \ \   \ \ \ \ \ \ \ \ \ \ \ \ \text{if $2\nmid dl-i$}\\
            \sum\limits_{j=0}^{\lfloor(dl-i)/(2d+2)\rfloor}(-1)^j\binom{l}{j}\binom
        {(dl-i)/2-j(d+1)+l-2}{l-2}
    \ \text{otherwise}
        \end{cases}.
    \end{equation}
\end{cor}
\begin{proof}
Since
    \[(1+x+x^2+\cdots+x^d)^{l}=(1-x^{d+1})^{l}(1-x)^{-l},\]
    the coefficients $K_{i,d,l}=A_{i,d,l}$ can be computed as
\begin{align*}
    A_{i,d,l}=&\sum_{j=0}^{\lfloor(dl-i)/(2d+2)\rfloor}(-1)^j\binom{l}{j}\binom{\lfloor(dl-i)/2\rfloor-j(d+1)+l-1}{l-1}\\
    &-\sum_{j=0}^{\lfloor(dl-i-1)/(2d+2)\rfloor}(-1)^j\begin{pmatrix}
        l\\j
    \end{pmatrix}\begin{pmatrix}
        \lfloor(dl-i-1)/2\rfloor-j(d+1)+l-1\\l-1
    \end{pmatrix}.
\end{align*}

    If $dl$ and $i$ have different parity,  $A_{i,d,l}=0$. If $dl$ and $i$ have the same parity, $\lfloor\frac{dl-i}{2}\rfloor=\frac{dl-i}{2}$ and $\lfloor\frac{dl-i-1}{2}\rfloor=\frac{dl-i}{2}-1$. Substituting these into the expression for $A_{i,d,l}$, and simplifying (via telescoping) gives the desired formula.
    \end{proof}

\begin{example}
The decomposition of $\left(\Sym^2 \rho_f\right)^{\otimes l}$ is given by
\[\left(\Sym^2 \rho_f\right)^{\otimes l}\cong\bigoplus \left(\Sym^i \rho_f\right)^{\oplus K_{i,2,l}}\]
where $K_{0,2,l}=0,1,1,3,6,15,36,91$ for $l=1,\cdots, 8$ respectively.
Additionally, the representation $\left(\Sym^d\rho_f\right)^{\otimes2}$ decomposes as
\[\left(\Sym^d\rho_f\right)^{\otimes2}\cong\bigoplus_{i=0}^{d}\Sym^{2i}\rho_f.\]
\end{example}
%\subsection{Relations between Dirichlet series}\label{section:dirichletseries}
%Let $\rho_i=\{\rho_{i,\ell}\}$, $i=1,2$ be compatible systems of $\ell-$adic representations, and $L(\rho_i,s)=\sum_{n=1}^\infty \frac{\lambda_{\rho_i}(n)}{n^s}$ be the associated $L-$functions. In one side, we can obtain new $\ell$-adic representations from the known representations with the associated $L$-functions, and on the other hand, we can obtain new Dirichlet series from the coefficients $\lambda_{\rho_i}(n)$. Now we consider when  the two constructions coincide. In this paper,we consider the following cases.
%\begin{itemize}
%    \item[(1)] The tensor product $L$-function $L(\rho_1\otimes\rho_2,s)$ and $D(\rho_1\otimes\rho_2,s):=\sum_{n=1}^\infty \frac{\lambda_{\rho_1}(n)\lambda_{\rho_2}(n)}{n^s}$.
%    \item[(2)] The symmetric power $L$-function $L(\Sym^k \rho,s)$ and $D(\Sym^k\rho,s):=\sum_{n=1}^{\infty}\frac{\lambda_\rho(n^k)}{n^s}$.
%\end{itemize}
%\begin{thm}\label{theorem:l-functions}
%    We have relations 
%    \begin{equation}
 %       L(\rho_1\otimes\rho_2,s)=D(\rho_1\otimes\rho_2,s)U_1(s),
  %  \end{equation}
  %  \begin{equation}
  %      L(\Sym^k\rho,s)=D(\Sym^k\rho,s)U_2(s),
  %  \end{equation}
  %  where $U_1$ and $U_2$ are absolutely convergent for $\Re(s)>1/2$ and uniformly convergent in the region $\Re(s)>1/2+\epsilon$ for any $\epsilon>0$.
%\end{thm}
%\begin{proof}
%    The proof is the same as case of Artin representations given in 
%\end{proof}

\subsection{Binary quadratic forms}
Let $Q(x,y)=ax^2+bxy+cy^2$ with $a,b,c\in\BZ$ and $a>0$ be a positive definite binary quadratic form. Let $D<0$ denote its discriminant. For any $n>0$, let $r(n,Q)$ denote the number of integer solutions $(x,y)$ of the equation $Q(x,y)=n$. In the classical case $Q(x,y)=x^2+y^2$,  the number $r(n,Q)$ counts the number of representations of $n$ as a sum of two squares.

Let $F=\BQ(\sqrt{D})$ be the imaginary quadratic field associated with $Q$. Denote by $\fh$  the class group of $F$ and by $h(D)$  its class number. Let $w_D$ be the number of  roots of unity in $F$, specially,   $w_D=6$ if $D=-3$, $w_D=4$ if $D=-4$ and  $w_D=2$ otherwise. The equivalent classes of binary quadratic forms are in bijection with the ideal class group $\fh$.In this correspondence, the number $r(n,Q)$ equals  $w_D$ times the number of integral ideals of norm $n$ in the  ideal class $\fa_Q$ corresponding to $Q$. 
 Applying the orthogonality relations of group characters, we obtain the formula
\begin{equation}\label{equation:binarydecom}
    r(n,Q)=\frac{w_D}{h(D)}\sum_{\chi}\overline{\chi(\fa_Q)}\sum_{\fn}\chi(\fn),
\end{equation}
where $\chi$ runs over all  characters of $\fh$ and the inner sum runs over all integral ideals $\fn$ of $F$ with norm  $n$. 

By class field theory, each  character of $\fh$ corresponds to  a Hecke character of $F$. For such a character $\chi$, let $V(\chi)=\Ind_F^\BQ\chi$ denote the  Artin representation induced from $F$ to $\BQ$. We note the following:
\begin{itemize}
    \item If $\chi$ is  trivial, then $V(\chi)\cong1_\BQ\oplus \chi_F$, where $\chi_F$ is the quadratic Dirichlet character associated to $F$.
    \item If $\chi$ is non-trivial, then $V(\chi)$ is an irreducible two-dimensional representation, which by the Langlands program corresponds to  a CM modular form $f_\chi=\sum_{n=1}^\infty a_\chi(n)q^n$.
\end{itemize}

%Let $$L(\chi,s)=\sum_{\fn}\frac{\chi(\fn)}{N(\fn)^s}=\sum_{n=1}^\infty\frac{a_\chi(n)}{n^s}$$
%be the Hecke $L$-function associated to $\chi$.

%Fixing an isomorphism $\BC\simeq \bar{\BQ_\ell}$ for each $\ell$, all these $L$-functions can be viewed as the $L$-function of an compatible systems $\ell$-adic representations. So the results in section \ref{section:dirichletseries} can be applied to these $L$-functions.

Now let $f$ be a cuspidal eigenform on $\SL_2(\BZ)$. For any integer $d\geq0$,  the symmetric power lifting $\Sym^d f$ is an automorphic form on $\GL_{d+1}$. For a character $\chi$ of $\fh$, we define the associated $L$-function as follows. If $\chi$ is non-trivial,
let
\[L\left(\Sym^d f, \chi,s\right):=\sum_{n=1}^\infty\frac{\lambda_{\Sym^df}(n)a_\chi(n)}{n^s}\]
be the Rankin-Selberg $L$-function associated to $\Sym^d f$ and $f_\chi$. If $\chi$ is trivial, define $$L\left(\Sym^d f, \chi,s\right)=L\left(\Sym^df,s\right)L\left(\Sym^df,\chi_F,s\right).$$ Similarly, we can define the $L$-functions $L\left(f^{\otimes l},\chi,s\right)$ and $L\left(\left(\Sym^d f\right)^{\otimes l},\chi,s\right)$ by incorporating the coefficients $a_\chi(n)$. Leveraging the tensor product decompositions from previous sections, we obtain the following factorizations:
 \begin{equation}
     L\left(f^{\otimes l},\chi,s\right)=\prod_{i=0}^{[l/2]}L\left(\Sym^{l-2i}f,\chi,s\right)^{c_{l,i}},
 \end{equation}
  \begin{equation}\label{equation:chilfactor}
   L\left(\left(\Sym^d f\right)^{\otimes l},\chi,s\right)=\prod_{i=0}^{dl}L\left(\Sym^i f,\chi,s\right)^{K_{i,d,l}}.   
  \end{equation}
\section{The estimations via Perron's formula}
\subsection{The Dirichlet series}
Let $f$ be a cuspidal eigenform of weight $k$ on $\SL_2(\BZ)$,  let $Q$ be a binary quadratic form and let $\chi$ be a character of $\fh$. We define the following Dirichlet series, which are central to our analysis:
\begin{align*}
&D\left(f^{\otimes l},s\right):=\sum_{n=1}^\infty \frac{\lambda_f^l(n)}{n^s},&D&\left(\left(\Sym^df\right)^{\otimes l},s\right):=\sum_{n=1}^\infty \frac{\lambda_{\Sym^df}^l(n)}{n^s},\\  
&D\left(f^{\otimes l},Q,s\right):=\sum_{n=1}^\infty \frac{\lambda_f^l(n)r(n,Q)}{n^s},&D&\left(\left(\Sym^d f\right)^{\otimes l},Q,s\right):=\sum_{n=1}^\infty \frac{\lambda_{\Sym^d f}^l(n)r(n,Q)}{n^s},\\
&D\left(f^{\otimes l},\chi,s\right):=\sum_{n=1}^\infty \frac{\lambda_f^l(n)a_{\chi}(n)}{n^s},&D&\left(\left(\Sym^d f\right)^{\otimes l},\chi,s\right):=\sum_{n=1}^\infty \frac{\lambda_{\Sym^d f}^l(n)a_\chi(n)}{n^s}.
\end{align*}
From the expression for $r(n,Q)$ given in (\ref{equation:binarydecom}), we immediately obtain the relations
\begin{align*}
   D\left(f^{\otimes l},Q,s\right)&=\frac{w_D}{h(D)}\sum_{\chi}\overline{\chi(\fa_Q)}D\left(f^{\otimes l},\chi,s\right),\\ 
   D\left(\left(\Sym^d f\right)^{\otimes l},Q,s\right)&=\frac{w_D}{h(D)}\sum_{\chi}\overline{\chi(\fa_Q)}D\left(\left(\Sym^d f\right)^{\otimes l},\chi,s\right).
\end{align*}
The connection between these Dirichlet series and the complete $L$-functions is provided by the following proposition, which is an analogue of Theorem \ref{theorem:l-functions} for the twisted case.

\begin{prop}\label{proposition:dirichletseriescomp}
The following identities hold:
\begin{align*}
   &D\left(f^{\otimes l},s\right)=L\left(f^{\otimes l},s\right)U_{f,l}(s),D\left(\left(\Sym^d f\right)^{\otimes l},s\right)=L\left(\left(\Sym^df\right)^{\otimes l},s\right)U_{\Sym^d f,l}(s),\\ 
   &D\left(f^{\otimes l},\chi,s\right)=L\left(f^{\otimes l},\chi,s\right)U_{f,l,\chi}(s),D\left(\left(\Sym^d f\right)^{\otimes l},\chi,s\right)=L\left(\left(\Sym^df\right)^{\otimes l},\chi,s\right)U_{\Sym^d f,l,\chi}(s).
\end{align*}
       where each  $U_*(s)$ is an Euler product that is absolutely convergent for  $\Re(s)>1/2$ and uniformly convergent in the region $\Re(s)>1/2+\epsilon$ for any $\epsilon>0$.
\end{prop}
\subsection{Subconvexity results for $L$-functions}
The proof of our main theorems relies on subconvexity estimates for various $L$-functions. We collect the necessary results here.
\begin{lem}[\cite{bourgain2017decoupling}]\label{lemma:subconvzeta}
    For any $\epsilon>0$ and any Dirichlet character $\chi$, we have
    \begin{equation}\label{equation:averagezeta}
      \zeta(\sigma+it)\ll(1+|t|)^{\frac{13}{42}(1-\sigma)+\epsilon}, \  L(\chi,\sigma+it)\ll(1+|t|)^{\frac{13}{42}(1-\sigma)+\epsilon} 
    \end{equation}
    uniformly for $1/2\leq\sigma\leq 1+\epsilon$ and $|t|\geq 1.$
   \end{lem}
\begin{lem}[\cite{good1982square}]
   Let $f$ be a modular form on $\SL_2(\BZ)$, for any $\epsilon>0$, we have
    \[L(f,\sigma+it)\ll(1+|t|)^{\frac{2}{3}(1-\sigma)+\epsilon}\]
    uniformly for $1/2\leq\sigma\leq 1+\epsilon$ and $|t|\geq 1.$
\end{lem}
\begin{lem}[\cite{dasgupta2024second}]
    Let $f$ be a modular form, for any $\epsilon>0$, we have
    \[L(\Sym^2 f,\sigma+it)\ll(1+|t|)^{\frac{8}{7}(1-\sigma)+\epsilon},\]
        uniformly for $1/2\leq\sigma\leq 1+\epsilon$ and $|t|\geq 1.$
\end{lem}

\begin{lem}[\cite{perelli1982general}]
    Let $L(s)$ be a general $L$-function of degree $m$, and assume that $L(s)$ converges absolutely for $\Re(s)>1$. Moreover, assume that $L(s)$ is entire except possibly simple poles at $s=0,1$, and it satisfies a functional equation of Riemann type. Then for any $\epsilon>0,$ we have
    \[L(\sigma+it)\ll (1+|t|)^{\frac{m}{2}(1-\sigma)+\epsilon}\]
    uniformly for $1/2\leq\sigma\leq 1+\epsilon $ and $|t|\geq 1$.
    And for $T\geq1$, we have
    \[\int_T^{2T}|L(1/2+it)|^2dt\ll T^{m/2+\epsilon}.\]
\end{lem}
\subsection{Proof of the main theorems}
The proofs of these theorems follow a standard approach via Perron's inversion formula, as detailed in \cite{Yang2024second} and \cite{Yang2024first}. For brevity, we present only the proof of Theorem \ref{thm:binary}, as the other cases are analogous.

We begin with the summation
\[S(x):=\sum_{Q(n_1,n_2)\leq x}\lambda^l_{\Sym^df}(Q(n_1,n_2))=\sum_{n\leq x}\lambda^l_{\Sym^df}(n)r(n,Q).\]
Using the expression for $r(n,Q)$ from (\ref{equation:binarydecom}), this becomes
\[S(x)= \frac{w_D}{h(D)}\sum_{\chi}\overline{\chi(\fa_Q)}\sum_{n\leq x}\lambda^l_{\Sym^df}(n)a_\chi(n),\]
where the sum runs over all characters of the ideal class group. It suffices to estimate the inner sum for each character  $\chi$.
By Perron's formula, for $b=1+\epsilon$ and a parameter $T$ with $1\leq T\leq x$, we have
\[\sum_{n\leq x}\lambda^l_{\Sym^df}(n)a_\chi(n)=\frac{1}{2\pi i}\int_{b-iT}^{b+iT}D\left(\left(\Sym^d f\right)^{\otimes l},\chi,s\right)\frac{x^s}{s}ds+O\left(\frac{x^b}{T}\right),\]
Shifting the contour of integration to the line  $\Re(s)=\frac{1}{2}+\epsilon$, an application of the Cauchy residue theorem yields
\[\sum_{n\leq x}\lambda^l_{\Sym^df}(n)a_\chi(n)=\frac{1}{2\pi i}\left\{\int_{\frac{1}{2}+\epsilon-iT}^{\frac{1}{2}+\epsilon+iT}+\int_{\frac{1}{2}+\epsilon+iT}^{b+iT}+\int_{b-iT}^{\frac{1}{2}+\epsilon-iT}\right\}D\left(\left(\Sym^d f\right)^{\otimes l},\chi,s\right)\frac{x^s}{s}ds\]
\[+\Res_{s=1}D\left(\left(\Sym^d f\right)^{\otimes l},\chi,s\right)\frac{x^s}{s}+\RO\left(\frac{x^{1+\epsilon}}{T}\right).\]
From Proposition \ref{proposition:dirichletseriescomp}, we have the factorization
\[D\left(\left(\Sym^d f\right)^{\otimes l},\chi,s\right)=L\left(\left(\Sym^df\right)^{\otimes l},\chi,s\right)U_{\Sym^d f,l,\chi}(s),\]
where $U_{\Sym^d f,l,\chi}(s)$ is holomorphic and bounded by $O(1)$ in $\Re(s)>1/2+\epsilon$.
The order of poles at $s=1$ is determined by the factorization (\ref{equation:chilfactor}):

$$\ord_{s=1}D\left(\left(\Sym^d f\right)^{\otimes l},\chi,s\right)=\ord_{s=1}L\left(\left(\Sym^d f\right)^{\otimes l},\chi,s\right)=\begin{cases}
    0\ \ \ \ \ \  \text{if $\chi$ is not trivial.}\\
    K_{0,d,l}\ \text{if $\chi$ is trivial.}
\end{cases}$$
Consequently, for the trivial character $\chi$ , the residue $\Res_{s=1}D\left(\left(\Sym^d f\right)^{\otimes l},\chi,s\right)\frac{x^s}{s}$ contributes a main term of the form $xP_{0,d,l}(\log x)$, where $P_{0,d,l}$ is a polynomial of degree $K_{0,d,l}-1.$ For non-trivial $\chi$, there is no pole at $s=1$, so the residue is 0.

Let us define the integrals along the different segments of the contour: 
\begin{align*}
    J_1=\int_{\frac{1}{2}+\epsilon-iT}^{\frac{1}{2}+\epsilon+iT}D\left(\left(\Sym^d f\right)^{\otimes l},\chi,s\right)\frac{x^s}{s}ds,\\
    J_2=\int_{\frac{1}{2}+\epsilon+iT}^{b+iT}D\left(\left(\Sym^d f\right)^{\otimes l},\chi,s\right)\frac{x^s}{s}ds,\\
    J_3=\int_{b-iT}^{\frac{1}{2}+\epsilon-iT}D\left(\left(\Sym^d f\right)^{\otimes l},\chi,s\right)\frac{x^s}{s}ds.
\end{align*}

We now proceed to estimate these integrals to determine the error terms.

\noindent\textbf{Estimate for $J_1$:}
If $\chi$ is  trivial,  we derive
\begin{align*}
    J_1&\ll x^{1/2+\epsilon}+x^{1/2+\epsilon}\int_{1}^T \left|L\left(\left(\Sym^df\right)^{\otimes l},\chi,1/2+\epsilon+it\right)\right|t^{-1}dt\\
&\ll x^{1/2+\epsilon}+x^{1/2+\epsilon}\log T\max_{T_1\leq T}\left\{T_1^{-1}\int_{T_1/2}^{T_1}\left|L\left(\left(\Sym^df\right)^{\otimes l},\chi,1/2+\epsilon+it\right)\right|dt\right\}\\
&\ll x^{1/2+\epsilon}+x^{1/2+\epsilon}T^{-1}T^{2\left(\frac{13}{84}K_{0,d,l}+\frac{1}{3}K_{1,d,l}+\frac{4}{7}K_{2,d,l}+\sum_{i=3}^\infty\frac{i+1}{4}K_{i,d,l}\right)+\epsilon}\\
&\ll x^{1/2+\epsilon}+x^{1/2+\epsilon}T^{\frac{(d+1)^l}{2}-\frac{4}{21}K_{0,d,l}-\frac{1}{3}K_{1,d,l}-\frac{5}{14}K_{2,d,l}-1+\epsilon},
\end{align*}
where we have used the identity $\sum\limits_{i=0}^\infty K_{i,d,l}=(d+1)^l$.
If $\chi$ is non-trivial, the estimate becomes
\begin{align*}
    J_1&\ll x^{1/2+\epsilon}+x^{1/2+\epsilon}T^{-1}T^{\frac{1}{3}K_{0,d,l}+\sum_{i=1}^\infty \frac{(i+1)}{2}K_{i,d,l}+\epsilon}\\
    &\ll x^{1/2+\epsilon}+x^{1/2+\epsilon}T^{\frac{(d+1)^l}{2}-\frac{1}{6}K_{0,d,l}-1+\epsilon}.
\end{align*}

\noindent\textbf{Estimate for $J_2+J_3$:} The horizontal integrals can be bounded by
\[J_2+J_3\ll\sup_{1/2+\epsilon\leq\sigma\leq 1+\epsilon}x^\sigma T^{-1}\left|L\left(\left(\Sym^d f\right)^{\otimes l},\chi,s\right)\right|.\]
Applying the decomposition formula (\ref{equation:chilfactor}) and the subconvexity bounds, we obtain the following estimates. 

For trivial $\chi$:
\begin{align*}
 J_2+J_3&\ll x^{\frac{1}{2}+\epsilon}T^{-1}T^{2\left(\frac{13}{84}K_{0,d,l}+\frac{1}{3}K_{1,d,l}+\frac{4}{7}K_{2,d,l}+\sum_{i=3}^\infty\frac{i+1}{4}K_{i,d,l}+\epsilon\right)}+\frac{x^{1+\epsilon}}{T}\\ 
 &=x^{1/2+\epsilon}T^{\frac{(d+1)^l}{2}-\frac{4}{21}K_{0,d,l}-\frac{1}{3}K_{1,d,l}-\frac{5}{14}K_{2,d,l}-1+\epsilon}+\frac{x^{1+\epsilon}}{T}.
\end{align*}

For non-trivial $\chi$:
\begin{align*}
 J_2+J_3&\ll x^{1/2+\epsilon}T^{-1}T^{\frac{1}{3}K_{0,d,l}+\sum_{i=1}^\infty \frac{(i+1)}{2}K_{i,d,l}+\epsilon}+\frac{x^{1+\epsilon}}{T}\\ 
 &=x^{1/2+\epsilon}T^{\frac{(d+1)^l}{2}-\frac{1}{6}K_{0,d,l}-1+\epsilon}+\frac{x^{1+\epsilon}}{T}.
\end{align*}

The parameter $T$  is now chosen to balance the error terms.
 We first consider the case where the class number of $Q$ is one. In this case, only the trivial character contributes. Choosing
$$T=x^{1/\left((d+1)^l-\frac{8}{21}K_{0,d,l}-\frac{2}{3}K_{1,d,l}-\frac{5}{7}K_{2,d,l}\right)},$$
we obtain the asymptotic formula
\[\sum_{n\leq x}\lambda^l_{\Sym^df}(n)a_\chi(n)=xP_{0,d,l}(\log(x))+O\left(x^{1-1/\left((d+1)^l-\frac{8}{21}K_{0,d,l}-\frac{2}{3}K_{1,d,l}-\frac{5}{7}K_{2,d,l}\right)+\epsilon}\right).\]
Next assume the class number is greater than 1. The main term comes only from the trivial character, but the error terms include contributions from all characters. Choosing
 $$T=x^{\frac{3}{3(d+1)^l-K_{0,d,l}}},$$ we obtain
\[\sum_{n\leq x}\lambda^l_{\Sym^df}(n)a_\chi(n)=xP_{0,d,l}(\log(x))+O\left(x^{1-\frac{3}{3(d+1)^l-K_{0,d,l}}+\epsilon}\right).\]
Summing over all characters  $\chi$, and using (\ref{equation:binarydecom}) completes the proof of Theorem \ref{thm:binary}. 
\bibliography{bib/ref}

@misc{arXiv:2601.17079,
 author = {Venkatasubbareddy, K.},
 title = {Generalization on the higher moments of the {Fourier} coefficients of symmetric power ${L}$-functions},
 year = {2026},
 howpublished = {Preprint, {arXiv}:2601.17079 [math.{NT}] (2026)},
 url = {https://arxiv.org/abs/2601.17079},
 arXiv = {arXiv:2601.17079}
}

@article{bourgain2017decoupling,
  author = {Bourgain, J.},
 title = {Decoupling, exponential sums and the {Riemann} zeta function},
 fjournal = {Journal of the American Mathematical Society},
 journal = {J. Amer. Math. Soc.},
 issn = {0894-0347},
 volume = {30},
 number = {1},
 pages = {205--224},
 year = {2017},
}

@article{dasgupta2024second,
  title={The second moment of the $ \mathrm{GL}_3 $ standard {L}-function on the critical line},
  author={Dasgupta, A. and Leung, W, H. and Young, M,P.},
  journal={arXiv preprint arXiv:2407.06962},
  year={2024}
}

@article{fomenko2006identities,
  title={Identities Involving Coefficients of Automorphic {L}-Functions.},
  author={Fomenko, O.M.},
  fjournal = {Journal of Mathematical Sciences },
journal = {J. Math. Sci.},
  volume={133},
  number={6},
  year={2006}
}

@article{fomenko2008mean,
  title={Mean value theorems for automorphic {L}-functions},
  author={Fomenko, O.M.},
  fjournal={St. Petersburg Mathematical Journal},
journal={St. Petersburg Math. J. },
  volume={19},
  number={5},
  pages={853--866},
  year={2008}
}

@book{fulton2013representation,
  title={Representation theory: a first course},
  author={Fulton, W. and Harris, J.},
  volume={129},
  year={2013},
  publisher={Springer Science \& Business Media}
}

@article{good1982square,
  title={The square mean of {D}irichlet series associated with cusp forms},
  author={Good, A.},
  journal={Mathematika},
  volume={29},
  number={2},
  pages={278--295},
  year={1982},
  publisher={London Mathematical Society}
}

@article{he2019integral,
   author = {He, X.},
 title = {Integral power sums of {Fourier} coefficients of symmetric square {{\(L\)}}-functions},
 fjournal = {Proceedings of the American Mathematical Society},
 journal = {Proc. Am. Math. Soc.},
 issn = {0002-9939},
 volume = {147},
 number = {7},
 pages = {2847--2856},
 year = {2019},
 language = {English},
 doi = {10.1090/proc/14516},
 zbMATH = {7073440},
 Zbl = {1431.11062}
}

@article{huang2021rankin,
  title={On the {R}ankin--{S}elberg problem},
  author={Huang, B.},
  journal={Math. Ann.},
  volume={381},
  number={3},
  pages={1217--1251},
  year={2021},
  publisher={Springer}
}

@article{kim2002functorial,
  title={Functorial products for $\mathrm{GL}_2\times \mathrm{GL}_3$ and the symmetric cube for $\mathrm{GL}_2$},
  author={Kim, H. H. and Shahidi, F.},
  journal={Ann. Math.},
  pages={837--893},
  year={2002},
  publisher={JSTOR}
}

@article{kim2002cuspidality,
  title={Cuspidality of symmetric powers with applications},
  author={Kim, H. H. and Shahidi, F.},
  journal={Duke Math. J.},
  volume={115},
  number={1},
  pages={177--197},
  year={2002}
}

@article{kim2003functoriality,
  title={Functoriality for the exterior square of $\mathrm{GL}_4$ and symmetric fourth of $\mathrm{GL}_2$, {A}ppendix 1 by {D}inakar {R}amakrishnan, {A}ppendix 2 by {H}enry {H}. {K}im and {P}eter {S}arnak},
  author={Kim, H.},
  journal={J. Amer. Math. Soc.},
  volume={16},
  pages={139--183},
  year={2003}
}

@inproceedings{kloosterman1927asymptotische,
  title={Asymptotische {F}ormeln f{\"u}r die {F}ourier-koeffizienten ganzer {M}odulformen},
  author={Kloosterman, H. D.},
  booktitle={ Abh. Math. Sem. Univ. Hamburg},
  volume={5},
  pages={337--352},
  year={1927},
  organization={Springer}
}

@article{lau2011integral,
  title={Integral power sums of {H}ecke eigenvalues},
  author={Lau, Y.-K. and L{\"u}, G. and Wu, J.},
  journal={ Acta Arith.},
  volume={150},
  number={2},
  pages={193--207},
  year={2011}
}

@article{liu2023asymptotic,
  title={On the asymptotic distribution of {F}ourier coefficients of cusp forms},
  author={Liu, H.},
  journal={Bull. Braz. Math. Soc. (N.S.) },
  volume={54},
  number={2},
  pages={21},
  year={2023},
  publisher={Springer}
}

@article{liu2023average,
  title={The Average Behavior of {F}ourier Coefficients of Symmetric Power {L}-Functions},
  author={Liu, H.},
  journal={Bull. Malays. Math. Sci. Soc. },
  volume={46},
  number={6},
  pages={193},
  year={2023},
  publisher={Springer}
}

@article{lu2009average,
  title={Average behavior of {F}ourier coefficients of cusp forms},
  author={L{\"u}, G.},
  journal={Proc. Am. Math. Soc.},
  volume={137},
  number={6},
  pages={1961--1969},
  year={2009}
}

@article{lu2011higher,
  title={On higher moments of {F}ourier coefficients of holomorphic cusp forms},
  author={L{\"u}, G.},
  journal={ Can. J. Math.},
  volume={63},
  number={3},
  pages={634--647},
  year={2011},
  publisher={Cambridge University Press}
}

@article{luo2021asymptotics,
  title={Asymptotics for the {D}irichlet coefficients of symmetric power {L}-functions},
  author={Luo, S. and Lao, H. and Zou, A.},
  journal={Acta Arith.},
  volume={199},
  pages={253--268},
  year={2021},
  publisher={Instytut Matematyczny Polskiej Akademii Nauk}
}

@article{newton2021symmetric1,
  author = {Newton, J. and Thorne, J. A.},
 title = {Symmetric power functoriality for holomorphic modular forms},
 fjournal = {Publications Math{\'e}matiques},
 journal = {Publ. Math., Inst. Hautes {\'E}tud. Sci.},
 issn = {0073-8301},
 volume = {134},
 pages = {1--116},
 year = {2021},
}

@article{newton2021symmetric2,
   author = {Newton, J. and Thorne, J. A.},
 title = {Symmetric power functoriality for holomorphic modular forms. {II}.},
 fjournal = {Publications Math{\'e}matiques},
 journal = {Publ. Math., Inst. Hautes {\'E}tud. Sci.},
 issn = {0073-8301},
 volume = {134},
 pages = {117--152},
 year = {2021},
}

@article{perelli1982general,
  title={General {L}-functions},
  author={Perelli, A.},
  journal={Ann. Mat. Pura Appl.},
  volume={130},
  pages={287--306},
  year={1982},
  publisher={Springer}
}

@inproceedings{rankin1939contributions,
  title={Contributions to the theory of {R}amanujan's function $\tau (n)$ and similar arithmetical functions:\uppercase\expandafter{\romannumeral 2} . {T}he order of the {F}ourier coefficients of integral modular forms},
  author={Rankin, R. A.},
  booktitle={Proc. Camb. Philos. Soc. },
  volume={35},
  number={3},
  pages={357--372},
  year={1939},
  organization={Cambridge University Press}
}

@article{rankin1983sums,
  title={Sums of powers of cusp form coefficients},
  author={Rankin, R. A.},
  journal={Math. Ann.},
  volume={263},
  pages={227--236},
  year={1983},
  publisher={Springer}
}

@article{selberg1940bemerkungen,
  title={Bemerkungen {\"u}ber eine {D}irichletsche {R}eihe, die mit der {T}heorie der {M}odulformen nahe verbunden ist},
  author={Selberg, A.},
  journal={Arch. Math. Naturvid.},
  volume={43},
  pages={47},
  year={1940}
}

@article{xu2022general,
  title={General asymptotic formula of {F}ourier coefficients of cusp forms over sum of two squares},
  author={Xu, C.},
  journal={J. Number Theory},
  volume={236},
  pages={214--229},
  year={2022},
  publisher={Elsevier}
}

@article{weil1948some,
  title={On some exponential sums},
  author={Weil, A.},
  journal={ Proc. Nat. Acad. Sci. U.S.A.},
  volume={34},
  number={5},
  pages={204--207},
  year={1948},
  publisher={National Acad Sciences}
}

@article{wu2009power,
  title={Power sums of {H}ecke eigenvalues and application},
  author={Wu, J.},
  journal={Acta Arith.},
  volume={137},
  number={4},
  pages={333--344},
  year={2009},
  publisher={Institute of Mathematics Polish Academy of Sciences}
}

@article{Yang2024first,
  title={Higher moments related to {D}edekind zeta functions of non-normal fields},
  author={Yang, J. and  Yang, Z.},
journal={J.Number theory},
  volume = {278},
pages = {547-569},
year = {2026},
}

@article{Yang2024second,
  title={Averages of coefficients of {D}edekind zeta functions over perfect powers},
  author={Yang,J. and Yang,Z.},
   year={2025},
journal={preprint},
}

@article{Yang2026higher,
 author = {Yang, J.},
 title = {Higher moments for non-normal fields with {Galois} group {{\(A_d\)}} and {{\(S_d\)}}},
 fjournal = {Journal of Number Theory},
 journal = {J. Number Theory},
 issn = {0022-314X},
 volume = {280},
 pages = {370--389},
 year = {2026},
 language = {English},
 doi = {10.1016/j.jnt.2025.08.018},
 zbMATH = {8113861}
}

@article{Yang2026second,
  title={Generating {D}irichlet series for arithmetic functions associated to {G}alois representations and applications},
  author={Yang,J.},
   year={2026},
journal={submitted},
}

@article{zhai2013average,
  title={Average behavior of {F}ourier coefficients of cusp forms over sum of two squares},
  author={Zhai, S.},
  journal={J. Number Theory},
  volume={133},
  number={11},
  pages={3862--3876},
  year={2013},
  publisher={Elsevier}
}
\bibliographystyle{plain}

\end{document}